\documentclass[11pt,a4paper]{article}

\usepackage[utf8]{inputenc}
\usepackage[T1]{fontenc}
\usepackage{amsmath, amssymb, amsthm}
\usepackage{geometry}
\usepackage{tikz-cd}
\usepackage{booktabs}
\usepackage{hyperref}
\usepackage[protrusion=true,expansion=false]{microtype}
\hypersetup{
    colorlinks=true,
    linkcolor=blue,
    citecolor=red,
    urlcolor=blue
}

\newtheorem{theorem}{Theorem}[section]
\newtheorem{lemma}[theorem]{Lemma}
\newtheorem{proposition}[theorem]{Proposition}
\newtheorem{corollary}[theorem]{Corollary}
\theoremstyle{definition}
\newtheorem{definition}[theorem]{Definition}
\newtheorem{example}[theorem]{Example}
\newtheorem{remark}[theorem]{Remark}
\newtheorem{question}[theorem]{Question}

\makeatletter
\newcommand{\address}[1]{\g@addto@macro\@thanks{\\#1}}
\newcommand{\email}[1]{\g@addto@macro\@thanks{\\\texttt{#1}}}
\makeatother

\newcommand{\UPF}{\mathrm{UPF}}
\newcommand{\DUPF}{\mathrm{DUPF}}

\newcommand{\PF}{\mathrm{PF}}
\newcommand{\OSP}{\mathrm{OSP}}
\newcommand{\DOSP}{\mathrm{DOSP}}
\newcommand{\Stir}[2]{\genfrac{\{}{\}}{0pt}{}{#1}{#2}}
\newcommand{\rStir}[3]{\genfrac{\{}{\}}{0pt}{}{#1}{#2}_{#3}}
\newcommand{\StirTwo}[2]{\genfrac{\{}{\}}{0pt}{}{#1}{#2}_{\geq 2}}
\newcommand{\cBell}{\widetilde{B}}
\newcommand{\DCP}{\mathrm{DCP}}
\newcommand{\Cay}{\mathrm{Cay}}
\newcommand{\std}{\operatorname{std}}
\newcommand{\fix}{\operatorname{fix}}

\title{On Deranged Unit-Interval Parking Functions\\ and the Deranged Bell Numbers}
\author{
    Yahia Djemmada\\[1ex]
    \small National Higher School of Mathematics \\[-0.5ex]
    \small Mahelma 16093, Sidi Abdellah, Algeria \\[1ex]
    \small USTHB, Faculty of Mathematics \\[-0.5ex]
    \small RECITS Laboratory, Algiers, Algeria \\[1ex]
    \normalsize \href{mailto:yahia.djemmada@nhsm.edu.dz}{\texttt{yahia.djemmada@nhsm.edu.dz}}
}
\date{\today}
\begin{document}

\maketitle

\begin{abstract}
Unit-interval parking functions are counted by the Fubini numbers and are in explicit bijection with ordered set partitions. We transport the deranged ordered set partitions of Belbachir, Djemmada, and N\'emeth through this bijection and obtain the deranged unit-interval parking functions $\DUPF_n$. The equality $|\DUPF_n|=\widetilde F_n$, the Stirling-transform formula, the exponential generating function $e^{1-e^x}/(2-e^x)$, and the dominant asymptotics are therefore not presented as new enumerative discoveries; they are consequences of the known deranged Bell-number theory. The new material of this note is the parking-side structure: leader and lucky-car characterizations, a fixed-block stratification of all unit-interval parking functions, rencontres-type generating functions and a Poisson limit law for fixed blocks, a bijective fixed-block decomposition of the Fubini numbers, a multivariate block-size refinement, a fully deranged $r$-start extension, and a Cayley-permutation model based on first appearances.
\end{abstract}

\section{Introduction}\label{sec:intro}

A \emph{preference list} $\alpha = (a_1, a_2, \dots, a_n) \in [n]^n$ is a \emph{parking function} if $n$ cars can all park on a one-way street with $n$ spots, where car $i$ drives to its preferred spot $a_i$ and parks in the first free spot at or after $a_i$. Konheim and Weiss~\cite{konheim1966} proved that there are $(n+1)^{n-1}$ parking functions of length $n$; we write $\PF_n$ for this set. Hadaway~\cite{hadaway2022} introduced the \emph{unit-interval parking functions} $\UPF_n \subseteq \PF_n$, those for which every car parks in its preferred spot or exactly one spot after it. Hadaway, and later Bradt et al.~\cite{bradt2024}, proved that these objects are enumerated by the Fubini numbers
\[
|\UPF_n| = F_n = \sum_{k=0}^{n} k!\,\Stir{n}{k},
\]
and Bradt et al. gave the block-structure description recalled below. Chaves Meyles et al.~\cite{chaves2023} made explicit a bijection between $\UPF_n$ and ordered set partitions of $[n]$ and related it to the face structure of the permutohedron.

Independently, Belbachir, Djemmada, and N\'emeth~\cite{belbachir2021} introduced the \emph{deranged Bell numbers} $\tilde F_n$. If an ordered set partition of $[n]$ is written as $B_{\sigma(1)} \mid \dots \mid B_{\sigma(k)}$, where $B_1, \dots, B_k$ are the same blocks listed by increasing minimum element, then it is called \emph{deranged} when $\sigma$ is a derangement of $[k]$. The number of such ordered set partitions is
\begin{equation}\label{eq:bdn-prop}
\tilde F_n = \sum_{k=0}^{n} d_k \Stir{n}{k},
\end{equation}
where $d_k$ is the number of derangements of $[k]$~\cite[Prop.~2.1]{belbachir2021}.

This note connects these two settings. Pulling the deranged ordered set partitions back through the bijection $\phi \colon \UPF_n \to \OSP_n$ gives the family
\[
\DUPF_n := \phi^{-1}(\DOSP_n),
\]
which we call the \emph{deranged unit-interval parking functions}. Since a bijection restricts to any preimage, the identity $|\DUPF_n|=\tilde F_n$ is immediate. It is therefore not the main theorem of the paper. The point of the paper is to ask what the deranged condition means intrinsically for unit-interval parking functions, and what refinements become natural on the parking side.

The results are organized as follows. Section~\ref{sec:prelim} fixes notation and gives a self-contained proof of the inverse of the bijection $\phi$. Section~\ref{sec:dupf} defines $\DUPF_n$, records the immediate enumeration, and gives the leader characterization and small examples. Section~\ref{sec:lucky} translates the condition into a criterion involving lucky cars. Section~\ref{sec:fixed} collects the enumerative refinements: the known EGF, the fixed-block stratification, the fixed-block probability generating function and Poisson limit law, a bijective fixed-block decomposition of the Fubini numbers, and the multivariate block-size refinement. Section~\ref{sec:r} treats the fully deranged $r$-start extension. Section~\ref{sec:cayley} gives the Cayley-permutation model, where the condition is read from the order of first appearances.

\paragraph{Novelty and relation to known results.}
We separate transported consequences from new parking-side structure. The sequence $\tilde F_n$ is OEIS entry A064898~\cite{oeis}, recorded by Penson as the Stirling transform of the derangement numbers; the ordered-set-partition interpretation, formula~\eqref{eq:bdn-prop}, the exponential generating function, and the basic asymptotics are due to Belbachir, Djemmada, and N\'emeth~\cite{belbachir2021}. Whenever a statement is a direct transfer through $\phi$, we say so. The additional content here is the parking-side description via leaders and lucky cars, the full fixed-block stratification of $\UPF_n$, the rencontres and Poisson refinements, the bijective fixed-block decomposition, the block-size refinement, and the fully deranged $r$-start and Cayley interpretations.

\paragraph{Related work.}
Restrictions of unit-interval parking functions and Fubini rankings driven by their ordered-set-partition interpretation are an active topic; Barreto et al.~\cite{barreto2025} study several such restricted families. On the partition side, Nkonkobe et al.~\cite{nkonkobe2020} studied barred preferential arrangements with no fixed blocks, a deranged analogue of the barred-preferential-arrangement model used for rational unit-interval parking functions in~\cite{aguilarfraga2024}; and the fixed-block point of view used here is consistent with the partial deranged Bell numbers of~\cite{djemmada2025}. The present note is a parking-function counterpart to these deranged-partition objects.

\section{Preliminaries and notation}\label{sec:prelim}

\subsection{Parking functions and conventions}

Throughout, $[n] = \{1, 2, \dots, n\}$ and $\alpha^\uparrow$ is the weakly increasing rearrangement of a tuple $\alpha$. For $\alpha \in \PF_n$, car $i$ has \emph{displacement} $s_i - a_i$, where $s_i$ is the spot in which car $i$ parks; the \emph{(total) displacement} of $\alpha$ is $D(\alpha) = \sum_i (s_i - a_i)$. By definition, $\alpha \in \UPF_n$ if and only if every car has displacement $0$ or $1$. Table~\ref{tab:notation} collects the notation; the objects it mentions are defined at first use.

\begin{table}[ht]
\centering
\small
\begin{tabular}{ll}
\toprule
$\PF_n,\ \UPF_n,\ \DUPF_n$ & parking functions; unit-interval; deranged unit-interval\\
$\OSP_n,\ \DOSP_n$ & ordered set partitions of $[n]$; deranged ordered set partitions\\
$\Cay_n,\ \DCP_n$ & Cayley permutations; deranged Cayley permutations\\
$F_n,\ \tilde F_n,\ B_n,\ \cBell_n,\ d_k$ & Fubini, deranged Bell, Bell, complementary Bell, derangement numbers\\
$\Stir{n}{k},\ \StirTwo{n}{k},\ \rStir{n}{k}{r}$ & Stirling numbers of the second kind: classical, $2$-associated, $r$-Stirling\\
$m(\alpha),\ D(\alpha)$ & number of blocks; total displacement of $\alpha$\\
$L(\alpha),\ \fix(\alpha)$ & set of lucky cars; number of fixed blocks of $\alpha$\\
$\nu_j,\ \ell_j,\ I_j,\ \mu_j$ & start, length, car set, and leader of the $j$-th block\\
$\phi,\ \psi,\ \rho,\ \Theta$ & the bijections of Sections~\ref{sec:prelim} and~\ref{sec:cayley}\\
\bottomrule
\end{tabular}
\caption{Notation used in this note.}
\label{tab:notation}
\end{table}

\paragraph{Conventions for $n = 0$.}
The empty tuple is the unique element of $\UPF_0$, and $\phi$ (defined below) sends it to the empty ordered set partition, which is vacuously deranged. Hence $\DUPF_0 = \UPF_0$ and $|\DUPF_0| = \tilde F_0 = 1$, consistent with the tables and generating functions below. A single-block partition is never deranged, so $\tilde F_1 = 0$.

\subsection{Block structure of unit-interval parking functions}

We recall the block-structure description of Bradt et al.~\cite{bradt2024}.

\begin{definition}[{\cite[Def.~2.7]{bradt2024}}]\label{def:blocks}
Let $\alpha \in \UPF_n$ and let $\alpha^\uparrow = (a^\uparrow_1, \dots, a^\uparrow_n)$. The \emph{block structure} of $\alpha$ is the factorization $\alpha^\uparrow = \pi_1 \mid \pi_2 \mid \dots \mid \pi_m$, where a new block $\pi_j$ begins at (and includes) each entry $a^\uparrow_i$ with $a^\uparrow_i = i$. We write $m = m(\alpha)$ for the number of blocks and $\ell_j = |\pi_j|$ for the length of $\pi_j$.
\end{definition}

The blocks act independently: the cars whose preferences lie in the value range of $\pi_j$ park exactly in the contiguous spots occupied by $\pi_j$, without interacting with other blocks~\cite[Obs.~2.8]{bradt2024}. The minimum value of block $\pi_j$ is $\nu_j := 1 + \sum_{l<j} \ell_l$, its value range is $[\nu_j, \nu_{j+1}-1]$ (with $\nu_{m+1} = n+1$), and each block equals the unique \emph{prime} unit-interval parking function of its length,
\begin{equation}\label{eq:prime}
\pi_j = (\nu_j,\ \nu_j,\ \nu_j + 1,\ \dots,\ \nu_j + \ell_j - 2)
\qquad (\text{just } (\nu_j) \text{ when } \ell_j = 1),
\end{equation}
which displaces $\ell_j - 1$ cars by one spot each~\cite{bradt2024}. Hence
\begin{equation}\label{eq:disp-blocks}
D(\alpha) = \sum_{j=1}^{m} (\ell_j - 1) = n - m(\alpha).
\end{equation}

The next theorem describes which rearrangements stay unit-interval; it is the key structural input.

\begin{theorem}[{\cite[Thm.~2.9]{bradt2024}}]\label{thm:bradt}
Let $\alpha \in \UPF_n$ have block structure $\pi_1 \mid \dots \mid \pi_m$. A rearrangement $\sigma$ of $\alpha$ lies in $\UPF_n$ if and only if, for each $j$, the entries of $\pi_j$ appear in $\sigma$ in the same relative order as in $\pi_j$. There are exactly $\binom{n}{\ell_1, \dots, \ell_m}$ such rearrangements.
\end{theorem}

\subsection{The ordered set partition of a unit-interval parking function}

We use the explicit bijection of Chaves Meyles et al.~\cite{chaves2023}, phrased in terms of which cars prefer which block. Because the whole note rests on this map, we include a complete proof of its inverse.

\begin{definition}\label{def:phi}
For $\alpha = (a_1, \dots, a_n) \in \UPF_n$ with block minima $\nu_1 < \dots < \nu_m$, assign to each car $i$ the block index
\[
\beta(i) := \text{the unique } j \text{ with } \nu_j \le a_i < \nu_{j+1},
\]
that is, the block containing the preference $a_i$, and set $I_j := \{ i \in [n] : \beta(i) = j \}$ for $j \in [m]$. The \emph{associated ordered set partition} of $\alpha$ is
\[
\phi(\alpha) := (I_1, I_2, \dots, I_m),
\]
whose blocks are listed in increasing block-value order. Since the cars preferring block $j$ are exactly those parking in its spots, $|I_j| = \ell_j \ge 1$, so $\phi(\alpha) \in \OSP_n$. This map agrees with~\cite[Def.~3.3]{chaves2023}.
\end{definition}

\begin{lemma}[Inverse construction]\label{lem:psi}
Let $(C_1, \dots, C_m) \in \OSP_n$ with $\ell_j = |C_j|$ and $\nu_j = 1 + \sum_{l<j} \ell_l$. Define $\psi(C_1, \dots, C_m) := \alpha = (a_1, \dots, a_n)$ as follows: for each $j$, list $C_j$ in increasing order as $c_{j,1} < c_{j,2} < \dots < c_{j,\ell_j}$ and set
\[
a_{c_{j,1}} = \nu_j,
\qquad
a_{c_{j,t}} = \nu_j + t - 2 \quad (2 \le t \le \ell_j),
\]
so that the cars of $C_j$, in increasing index order, receive the prime preferences~\eqref{eq:prime}. Then:
\begin{enumerate}
\item[\textup{(i)}] $\alpha \in \UPF_n$. More precisely, car $c_{j,1}$ parks in spot $\nu_j$ with displacement $0$, and car $c_{j,t}$ parks in spot $\nu_j + t - 1$ with displacement $1$ for $2 \le t \le \ell_j$; so the cars of $C_j$ fill exactly the spots $\nu_j, \dots, \nu_j + \ell_j - 1$, and each block contains exactly one undisplaced car, namely $\min C_j$.
\item[\textup{(ii)}] The block structure of $\alpha$ is $\pi_1 \mid \dots \mid \pi_m$ with $\pi_j$ as in~\eqref{eq:prime}, and the cars preferring block $j$ are exactly those of $C_j$. Hence $\phi(\psi(C_1, \dots, C_m)) = (C_1, \dots, C_m)$.
\item[\textup{(iii)}] $\psi(\phi(\alpha)) = \alpha$ for every $\alpha \in \UPF_n$.
\end{enumerate}
\end{lemma}

\begin{proof}
(i) All preferences of the cars of $C_j$ lie in $[\nu_j, \nu_j + \ell_j - 2] \subseteq [\nu_j, \nu_{j+1} - 1]$, and these value ranges are pairwise disjoint. We show by induction on the arrival index $i$ that the following invariant holds after cars $1, \dots, i$ have parked: for every $j$, the parked cars of $C_j \cap [i]$ occupy exactly the spots $\nu_j, \nu_j + 1, \dots, \nu_j + |C_j \cap [i]| - 1$. The invariant is empty for $i = 0$. Suppose it holds after car $i - 1$, and let car $i = c_{j,t}$ arrive. By the invariant, every parked car sits in the spot range of its own block, so the occupied spots inside $[\nu_j, \nu_{j+1}-1]$ are exactly $\nu_j, \dots, \nu_j + t - 2$, filled by $c_{j,1}, \dots, c_{j,t-1}$. If $t = 1$, spot $a_i = \nu_j$ is free and car $i$ parks there with displacement $0$. If $t \ge 2$, spot $a_i = \nu_j + t - 2$ is occupied, and the next spot $\nu_j + t - 1 \le \nu_j + \ell_j - 1 < \nu_{j+1}$ is free, so car $i$ parks there with displacement $1$. In both cases the invariant is preserved. Hence every car parks with displacement at most $1$, and $\alpha \in \UPF_n$.

(ii) By construction, $\alpha^\uparrow$ consists, for each $j$ in turn, of the entries $\nu_j, \nu_j, \nu_j + 1, \dots, \nu_j + \ell_j - 2$ (just $\nu_j$ if $\ell_j = 1$). Exactly $\nu_j - 1$ entries are smaller than $\nu_j$, namely those of blocks $1, \dots, j-1$, so $a^\uparrow_{\nu_j} = \nu_j$; and for $\nu_j < i < \nu_{j+1}$ one checks $a^\uparrow_i = i - 1 < i$. By Definition~\ref{def:blocks}, the blocks of $\alpha$ begin exactly at the positions $\nu_1, \dots, \nu_m$, and $\pi_j$ is~\eqref{eq:prime}. The value range of block $j$ is $[\nu_j, \nu_{j+1}-1]$, and the cars whose preference lies in it are exactly those of $C_j$; hence $I_j = C_j$.

(iii) Let $\alpha \in \UPF_n$ with block structure $\pi_1 \mid \dots \mid \pi_m$ and car sets $I_1, \dots, I_m$, where $\pi_j$ is the prime sequence~\eqref{eq:prime}. The tuple $\alpha$ is a unit-interval rearrangement of $\alpha^\uparrow$, so by Theorem~\ref{thm:bradt} the entries of $\pi_j$ occur in $\alpha$ in exactly the relative order~\eqref{eq:prime}. These entries are carried by the cars of $I_j$, read in increasing index order; so the $t$-th smallest car of $I_j$ has preference equal to the $t$-th entry of $\pi_j$. This is exactly the rule defining $\psi$ applied to $(I_1, \dots, I_m) = \phi(\alpha)$; hence $\psi(\phi(\alpha)) = \alpha$.
\end{proof}

Note that the rule for $\psi$ involves no hidden choices: the only repeated preference inside a block is $\nu_j$, given to the two smallest cars of $C_j$, and only their relative order matters.

\begin{theorem}[{\cite[Thm.~3.6]{chaves2023}}]\label{thm:phi-bij}
The map $\phi \colon \UPF_n \to \OSP_n$ is a bijection, with inverse $\psi$.
\end{theorem}

\begin{proof}
Immediate from Lemma~\ref{lem:psi}(ii) and (iii).
\end{proof}

\subsection{Deranged ordered set partitions}

\begin{definition}[{\cite[Def.~1]{belbachir2021}}]\label{def:deranged}
Let $(C_1, \dots, C_k)$ be an ordered set partition of $[n]$, and let $B_1, \dots, B_k$ be the same blocks listed by increasing minimum element, so $\min B_1 < \dots < \min B_k$. We say $(C_1, \dots, C_k)$ is \emph{deranged} if $C_i \ne B_i$ for all $i \in [k]$; equivalently, the permutation $\sigma \in S_k$ with $C_i = B_{\sigma(i)}$ is a derangement. We write $\DOSP_n$ for the set of deranged ordered set partitions of $[n]$; by~\eqref{eq:bdn-prop}, $|\DOSP_n| = \tilde F_n$.
\end{definition}

\section{Deranged unit-interval parking functions}\label{sec:dupf}

\begin{definition}\label{def:dupf}
A unit-interval parking function $\alpha \in \UPF_n$ is a \emph{deranged unit-interval parking function} if $\phi(\alpha)$ is deranged. We write
\[
\DUPF_n := \{ \alpha \in \UPF_n : \phi(\alpha) \in \DOSP_n \} = \phi^{-1}(\DOSP_n).
\]
\end{definition}

Because $\DUPF_n$ is defined as a preimage, its enumeration is immediate. We record it for reference; the content of this note lies in the finer statements that follow.

\begin{proposition}\label{prop:count}
For all $n \ge 0$, the map $\phi$ restricts to a bijection $\DUPF_n \to \DOSP_n$, and
\[
|\DUPF_n| = \tilde F_n = \sum_{k=0}^{n} d_k \Stir{n}{k}.
\]
\end{proposition}

\begin{proof}
A bijection restricts to a bijection between any preimage and its image; take the preimage $\phi^{-1}(\DOSP_n) = \DUPF_n$ in Theorem~\ref{thm:phi-bij} and apply~\eqref{eq:bdn-prop}.
\end{proof}

The deranged condition has a description that does not mention $\phi$: it compares two natural orders on the blocks of $\alpha$.

\begin{proposition}[Intrinsic characterization]\label{prop:intrinsic}
For $\alpha \in \UPF_n$ with associated blocks $I_1, \dots, I_m$ (in block-value order, as in Definition~\ref{def:phi}), call $\mu_j := \min I_j$ the \emph{leader} of block $j$: the least index of a car whose preference lies in block $j$. Then $\alpha \in \DUPF_n$ if and only if, for every $j \in [m]$, the leader $\mu_j$ is \emph{not} the $j$-th smallest among $\mu_1, \dots, \mu_m$.
\end{proposition}

\begin{proof}
The minimum element of $I_j$ is $\mu_j$. List the blocks by increasing minimum element as $B_1, \dots, B_m$, so $B_p$ is the block whose leader is the $p$-th smallest of $\mu_1, \dots, \mu_m$. By Definition~\ref{def:deranged}, $\alpha \in \DUPF_n$ iff $I_j \ne B_j$ for all $j$; and $I_j = B_j$ exactly when $\mu_j$ is the $j$-th smallest leader.
\end{proof}

In words, $\alpha$ is deranged exactly when sorting its blocks by value gives a derangement of the order obtained by sorting the blocks by their least-indexed car. This is nothing more than Definition~\ref{def:deranged} rewritten in parking language; we use it constantly below. For a single block the two orders coincide, so single-block parking functions are never deranged.

For later use, we also record the corresponding fixed-block terminology.

\begin{definition}\label{def:fixedblock}
Let $(C_1, \dots, C_k) \in \OSP_n$ with blocks $B_1, \dots, B_k$ in increasing min-element order. The block $C_i$ is a \emph{fixed block} if $C_i = B_i$. For $\alpha \in \UPF_n$, a value-block $I_i$ of $\alpha$ is \emph{fixed} if it is a fixed block of $\phi(\alpha)$; by Proposition~\ref{prop:intrinsic}, this happens exactly when the leader $\mu_i$ is the $i$-th smallest leader. We write $\fix(\alpha)$ for the number of fixed blocks of $\alpha$; thus $\alpha \in \DUPF_n$ if and only if $\fix(\alpha) = 0$.
\end{definition}

\subsection{Examples, values, and the permutohedron}\label{sec:examples}

\begin{example}[$n = 2$]
We have $\UPF_2 = \{(1,1), (1,2), (2,1)\}$. For $(1,2)$, $\phi(1,2) = (\{1\}, \{2\})$ is the identity ordering; both blocks are fixed. For $(1,1)$, there is a single block and $\phi(1,1) = (\{1,2\})$ is fixed. For $(2,1)$, car $1$ prefers block $2$ and car $2$ prefers block $1$, so $\phi(2,1) = (\{2\}, \{1\})$: both blocks move. Thus $\DUPF_2 = \{(2,1)\}$ and $|\DUPF_2| = 1 = \tilde F_2$.
\end{example}

\begin{example}[$n = 3$]\label{ex:n3}
Among the $13$ elements of $\UPF_3$, the deranged ones are
\[
\DUPF_3 = \{\, (2,1,2),\ (2,2,1),\ (3,1,1),\ (2,3,1),\ (3,1,2) \,\},
\]
so $|\DUPF_3| = 5 = \tilde F_3$. They split by displacement as Corollary~\ref{cor:refine} predicts: the three functions of displacement $1$ (two blocks), namely $(2,1,2), (2,2,1), (3,1,1)$, give $d_2 \Stir{3}{2} = 1 \cdot 3 = 3$; the two derangement permutations $(2,3,1), (3,1,2)$ of displacement $0$ (three blocks) give $d_3 \Stir{3}{3} = 2 \cdot 1 = 2$. For instance $\phi(3,1,1) = (\{2,3\}, \{1\})$, whose min-element order is $(\{1\}, \{2,3\})$: both blocks move. For a permutation every block is a singleton and $\phi(\alpha) = \alpha$, so the permutations in $\DUPF_n$ are exactly the derangements.
\end{example}

\begin{table}[ht]
\centering
\begin{tabular}{lrrrrrrrrr}
\toprule
$n$ & $0$ & $1$ & $2$ & $3$ & $4$ & $5$ & $6$ & $7$ & $8$ \\
\midrule
$|\UPF_n| = F_n$   & $1$ & $1$ & $3$  & $13$ & $75$  & $541$ & $4683$ & $47293$ & $545835$ \\
$|\DUPF_n| = \tilde F_n$ & $1$ & $0$ & $1$  & $5$  & $28$  & $199$ & $1721$ & $17394$ & $200803$ \\
\bottomrule
\end{tabular}
\caption{Unit-interval and deranged unit-interval parking functions. The values of $|\DUPF_n|$ for $n \le 7$ were confirmed by exhaustive search and agree with the deranged Bell numbers $\tilde F_n$ of~\cite{belbachir2021}, OEIS entry A064898~\cite{oeis}.}
\label{tab:values}
\end{table}

\begin{remark}[Permutohedron]\label{rmk:perm}
Chaves Meyles et al.~\cite{chaves2023} identify the unit-interval parking functions of length $n$ with $m$ blocks with the $(n-m)$-dimensional faces of the permutohedron of order $n$, faces being labelled by ordered set partitions. Under this labelling, $\DUPF_{n,m}$ picks out a \emph{subset} of $d_m \Stir{n}{m}$ faces in dimension $n - m$. This subset is not closed under passing to faces, so the deranged faces do not form a subcomplex: already for $n = 3$, the edge labelled $(\{2,3\}, \{1\})$ is deranged, while one of its two vertices, $(\{3\}, \{2\}, \{1\})$, is not (its middle block is fixed). The deranged faces do carry the finer stratification of Theorem~\ref{thm:strat} by the number of fixed blocks; we do not pursue the geometry here.
\end{remark}

\subsection{Characterization through lucky cars}\label{sec:lucky}

Proposition~\ref{prop:intrinsic} tests the deranged condition through the block structure. We now reduce it to a single derangement test on a word, and then to a comparison of two orderings of the lucky cars.

\begin{definition}\label{def:leaderword}
For $\alpha \in \UPF_n$ with value-blocks $I_1, \dots, I_m$ and leaders $\mu_j = \min I_j$, the \emph{leader word} of $\alpha$ is the sequence $W(\alpha) = (\mu_1, \dots, \mu_m)$ of distinct car indices. Its \emph{standardization} $\std(W(\alpha)) \in S_m$ is the unique permutation order-isomorphic to $W(\alpha)$: $\std(W(\alpha))(j)$ is the rank of $\mu_j$ among $\mu_1, \dots, \mu_m$ sorted increasingly.
\end{definition}

\begin{proposition}\label{prop:std}
$\alpha \in \DUPF_n$ if and only if $\std(W(\alpha))$ is a derangement in $S_m$.
\end{proposition}

\begin{proof}
$\std(W(\alpha))$ fixes $j$ exactly when $\mu_j$ is the $j$-th smallest leader, which by Proposition~\ref{prop:intrinsic} is exactly when block $I_j$ is fixed.
\end{proof}

\begin{definition}\label{def:lucky}
For $\alpha \in \UPF_n$, a car $i \in [n]$ is \emph{lucky} if it parks in its preferred spot $a_i$ (displacement $0$). We write $L(\alpha) := \{ i \in [n] : \text{car } i \text{ is lucky} \}$.
\end{definition}

By Lemma~\ref{lem:psi}(i) and (iii), each value-block contains exactly one lucky car, namely its leader, and the leader's preference is the block start; conversely every leader is lucky (cf.~\cite{bradt2024}). Hence $|L(\alpha)| = m(\alpha)$ and $L(\alpha) = \{\mu_1, \dots, \mu_m\}$.

\begin{theorem}[Characterization via lucky cars]\label{thm:local}
Let $\alpha \in \UPF_n$ and $m = |L(\alpha)|$. List the lucky cars in two orders:
\begin{itemize}
    \item by increasing car index: $u_1 < u_2 < \dots < u_m$;
    \item by increasing preference value: $w_1, w_2, \dots, w_m$, so that $a_{w_1} < a_{w_2} < \dots < a_{w_m}$.
\end{itemize}
(The preferences of the lucky cars are the distinct block starts $\nu_1 < \dots < \nu_m$, so the second order is well defined.) Then $\alpha \in \DUPF_n$ if and only if $u_k \neq w_k$ for all $k \in [m]$.
\end{theorem}

\begin{proof}
The lucky car whose preference is $\nu_k$ is the leader of block $k$, so $w_k = \mu_k$; and $u_k$ is the $k$-th smallest leader. By Proposition~\ref{prop:intrinsic}, block $k$ is fixed if and only if its leader $\mu_k = w_k$ is the $k$-th smallest leader, that is, $w_k = u_k$. So $\alpha$ is deranged if and only if $u_k \neq w_k$ for all $k$.
\end{proof}

The test needs only the outcome of the parking process: run the cars, record the lucky ones, and compare the two orders. No standardization step is required, though the test is of course not free of the parking simulation. It also shows that the deranged condition is not a coordinate condition of the form ``$a_i \neq i$'': the test fails in both directions, since $(2,2,1) \in \DUPF_3$ although $a_2 = 2$, while $(4,3,1,1) \in \UPF_4 \setminus \DUPF_4$ although $a_i \neq i$ for every $i$ (its block $\{2\}$ is fixed in second position).

\begin{corollary}[The first car]\label{cor:a1}
For every $\alpha \in \UPF_n$ with $n \ge 1$, car $1$ parks in its preferred spot; so car $1$ is \emph{always} lucky, and it is the leader of the block containing $a_1$. If $\alpha \in \DUPF_n$, then this block is not the first block; equivalently $a_1 > \ell_1$, and in particular $a_1 \neq 1$. Conversely, if $a_1 = 1$ then the first block is fixed and $\alpha \notin \DUPF_n$. The condition $a_1 \neq 1$ alone is not sufficient: $(2,1,3) \in \UPF_3 \setminus \DUPF_3$.
\end{corollary}

\begin{proof}
Car $1$ arrives first, when the street is empty, so it parks in spot $a_1$ and is lucky. Having the least possible index, it is the smallest lucky car, so $u_1 = 1$; and it is the leader of the block $j_0 := \beta(1)$ containing $a_1$, so $w_{j_0} = 1$. Block $j_0$ is fixed iff $w_{j_0} = u_{j_0}$, i.e.\ iff $u_{j_0} = 1$, i.e.\ iff $j_0 = 1$, because the $u_k$ increase strictly from $u_1 = 1$. So if $\alpha$ is deranged then $j_0 \ge 2$, whence $a_1 \ge \nu_2 = \ell_1 + 1 \ge 2$; and if $a_1 = 1 = \nu_1$ then $j_0 = 1$ and the first block is fixed. Finally, $\phi(2,1,3) = (\{2\}, \{1\}, \{3\})$ has the fixed block $\{3\}$ in third position, so $(2,1,3) \notin \DUPF_3$ although $a_1 = 2$.
\end{proof}

The example $(2,2,1) \in \DUPF_3$ shows that nothing stronger holds for car $1$: there, car $1$ is lucky with $a_1 = 2$, leading the second block, and $\alpha$ is deranged. In fact, being the first to arrive, car $1$ is lucky in \emph{every} parking function; the deranged condition can only control \emph{which} block car $1$ leads, never whether it is lucky.

\section{Fixed-block and generating-function refinements}\label{sec:fixed}

This section gathers the enumerative consequences and refinements. We use the exponential generating functions
\[
F(x) = \sum_{n\ge0} F_n \frac{x^n}{n!} = \frac{1}{2 - e^x},
\qquad
B(x) = \sum_{n\ge0} B_n \frac{x^n}{n!} = e^{e^x - 1},
\qquad
D(z) = \sum_{k\ge0} d_k \frac{z^k}{k!} = \frac{e^{-z}}{1-z}.
\]
Because $|\DUPF_n| = \tilde F_n$ by Proposition~\ref{prop:count}, every known statement about the deranged Bell numbers also applies to $\DUPF_n$; the point here is to record the parking-side refinements that come from fixed blocks and block sizes.

\subsection{The exponential generating function}

The generating function of $\tilde F_n$ is due to Belbachir, Djemmada, and N\'emeth~\cite[Thm.~3.1]{belbachir2021}; it is not new here. We include a proof by the symbolic method because the composition it uses---a derangement of nonempty blocks---is exactly the structure that $\DUPF_n$ carries through $\phi$, and because the refinements below (Proposition~\ref{prop:fixegf} and Theorem~\ref{thm:multi}) are obtained by marking this same composition.

\begin{theorem}[{\cite[Thm.~3.1]{belbachir2021}}]\label{thm:egf}
The exponential generating function of the deranged unit-interval parking functions is
\[
\tilde F(x) := \sum_{n\ge0} |\DUPF_n|\,\frac{x^n}{n!}
            = \sum_{n\ge0} \tilde F_n\,\frac{x^n}{n!}
            = \frac{e^{1 - e^x}}{2 - e^x}.
\]
\end{theorem}

\begin{proof}[Proof by the symbolic method]
A deranged ordered set partition of $[n]$ is a structure of the labelled composition $\mathrm{Der} \circ \mathcal{E}_{\ge1}$ of two species: each part is a nonempty set $\mathcal{E}_{\ge1}$, with generating function $E(x) = e^x - 1$, and the collection of parts carries a derangement structure $\mathrm{Der}$, with generating function $D(z) = e^{-z}/(1-z)$. Indeed, fix a set partition of $[n]$ into $k$ blocks and list them by increasing minimum element as $B_1, \dots, B_k$. A $\mathrm{Der}$-structure on the set of blocks is a fixed-point-free permutation of $\{B_1, \dots, B_k\}$; under the labelling $B_i \leftrightarrow i$ it is a derangement $\sigma \in S_k$ and encodes the deranged ordering $B_{\sigma(1)} \mid \dots \mid B_{\sigma(k)}$ of Definition~\ref{def:deranged}, the labelling guaranteeing that ``$\sigma$ has no fixed point'' means ``no block sits in its own position''. By the composition theorem for labelled structures~\cite[Thm.~II.2]{flajolet2009},
\[
\tilde F(x) = D\bigl(E(x)\bigr)
            = \frac{e^{-(e^x - 1)}}{1 - (e^x - 1)}
            = \frac{e^{1-e^x}}{2 - e^x}.
\]
Extracting coefficients recovers~\eqref{eq:bdn-prop}: since $(e^x-1)^k/k! = \sum_{n\ge0} \Stir{n}{k}\, x^n/n!$,
\[
D\bigl(E(x)\bigr)
= \sum_{k\ge0} \frac{d_k}{k!}(e^x-1)^k
= \sum_{n\ge0}\Bigl(\sum_{k=0}^{n} d_k \Stir{n}{k}\Bigr)\frac{x^n}{n!}. \qedhere
\]
\end{proof}

\begin{remark}\label{rmk:asym}
The function $\tilde F(x)$ is meromorphic; its dominant singularity is the simple pole $x_0 = \log 2$ (where $e^x = 2$; all other solutions have larger modulus), and the entire numerator $e^{1-e^x}$ equals $e^{-1}$ there. Singularity analysis~\cite[Ch.~IV]{flajolet2009} yields
\[
\frac{\tilde F_n}{n!} \sim \frac{1}{2e\,(\log 2)^{n+1}},
\]
which is~\cite[Thm.~5.1]{belbachir2021} and is also recorded in the OEIS entry for this sequence~\cite{oeis}. In particular $|\DUPF_n|$ grows at the same exponential rate $(\log 2)^{-n}$ as $|\UPF_n| = F_n \sim n!/\bigl(2(\log 2)^{n+1}\bigr)$, and $|\DUPF_n|/|\UPF_n| \to e^{-1}$. Corollary~\ref{cor:poisson} below refines this last limit.
\end{remark}

\subsection{Fixed-block stratification}\label{subsec:fixed-strat}

Rather than only removing the ordered set partitions with a fixed block, we can count unit-interval parking functions by how many fixed blocks they have. The deranged family is the bottom stratum.

\begin{theorem}[Fixed-block stratification]\label{thm:strat}
For $0 \le r \le m \le n$,
\[
\bigl| \{ \alpha \in \UPF_n : m(\alpha) = m,\ \fix(\alpha) = r \} \bigr|
= \Stir{n}{m} \binom{m}{r} d_{m-r}.
\]
\end{theorem}

\begin{proof}
The map $\phi$ carries the blocks of $\alpha$ to the blocks of $\phi(\alpha)$, so it preserves both $m(\alpha)$ and $\fix(\alpha)$; it therefore suffices to count the ordered set partitions of $[n]$ with $m$ blocks, exactly $r$ of them fixed. Choose the underlying unordered partition $P$ ($\Stir{n}{m}$ ways) and write its blocks in increasing min-element order as $B_1, \dots, B_m$. Every ordering of $P$ is $(B_{\sigma(1)}, \dots, B_{\sigma(m)})$ for a unique $\sigma \in S_m$, and the block in position $i$ is fixed exactly when $\sigma(i) = i$. So the orderings with exactly $r$ fixed blocks correspond to the permutations of $[m]$ with exactly $r$ fixed points, of which there are $\binom{m}{r} d_{m-r}$: choose the fixed points, then derange the rest.
\end{proof}

\begin{corollary}[Refinement by displacement]\label{cor:refine}
For $1 \le m \le n$, the number of $\alpha \in \DUPF_n$ with exactly $m$ blocks is
\[
|\DUPF_{n,m}| = d_m \Stir{n}{m},
\]
and by~\eqref{eq:disp-blocks} this is also the number of $\alpha \in \DUPF_n$ with displacement $D(\alpha) = n - m$. Summing over $m$ recovers Proposition~\ref{prop:count}; summing Theorem~\ref{thm:strat} over $r$ instead recovers $m!\,\Stir{n}{m}$ and $|\UPF_n| = F_n$.
\end{corollary}

\begin{proof}
Take $r = 0$ in Theorem~\ref{thm:strat} and use $D(\alpha) = n - m(\alpha)$.
\end{proof}

On the ordered-set-partition side, Theorem~\ref{thm:strat} is an elementary rencontres count; the parking content is the identity $D(\alpha) = n - m(\alpha)$, which turns the block count into a displacement statistic, and the lucky-car reading of fixed blocks in Section~\ref{sec:lucky}.

\subsection{Marking fixed blocks and the Poisson limit law}

The stratification of Theorem~\ref{thm:strat} packages into a rencontres-type generating function.

\begin{proposition}[Marking fixed blocks]\label{prop:fixegf}
Let $R_m(u) := \sum_{r=0}^{m} \binom{m}{r} d_{m-r}\, u^r$ be the rencontres polynomial, which counts the permutations of $[m]$ by number of fixed points. Then
\[
\sum_{\alpha \in \UPF_n} u^{\fix(\alpha)}\, t^{\,D(\alpha)}
= \sum_{m=0}^{n} \Stir{n}{m}\, t^{\,n-m}\, R_m(u),
\qquad
\sum_{n \ge 0}\ \sum_{\alpha \in \UPF_n} u^{\fix(\alpha)}\, \frac{x^n}{n!}
= \frac{e^{(u-1)(e^x-1)}}{2 - e^x}.
\]
Setting $u = 0$ gives $\tilde F(x)$ (Theorem~\ref{thm:egf}); setting $u = 1$ gives $F(x) = 1/(2 - e^x)$.
\end{proposition}

\begin{proof}
The first identity restates Theorem~\ref{thm:strat}, using $D(\alpha) = n - m(\alpha)$ from~\eqref{eq:disp-blocks}. For the second, $\sum_{m \ge 0} R_m(u)\, z^m/m! = e^{uz}\, D(z) = e^{(u-1)z}/(1-z)$, because a permutation counted by fixed points is a set of fixed points together with a derangement of the remaining elements. Substituting $z = e^x - 1$, the labelled composition with nonempty blocks used in the proof of Theorem~\ref{thm:egf}, gives the stated function.
\end{proof}

\begin{corollary}[Poisson limit law for fixed blocks]\label{cor:poisson}
Let $\alpha$ be chosen uniformly at random from $\UPF_n$. Then $\fix(\alpha)$ converges in distribution, as $n \to \infty$, to a Poisson random variable of mean $1$:
\[
\Pr\bigl[\fix(\alpha) = r\bigr] \longrightarrow \frac{e^{-1}}{r!} \qquad (r = 0, 1, 2, \dots).
\]
The case $r = 0$ is the limit $\tilde F_n / F_n \to e^{-1}$ of Remark~\ref{rmk:asym}.
\end{corollary}

\begin{proof}
Fix $u \in [0,1]$ and let $G_u(x) := e^{(u-1)(e^x-1)}/(2 - e^x)$ as in Proposition~\ref{prop:fixegf}. As in Remark~\ref{rmk:asym}, $G_u$ has a unique dominant singularity at $x_0 = \log 2$, a simple pole, near which the entire factor $e^{(u-1)(e^x-1)}$ tends to $e^{u-1}$ and $2 - e^x = 2(x_0 - x)\bigl(1 + O(x_0 - x)\bigr)$. Singularity analysis gives $n!\,[x^n]\, G_u(x) \sim \tfrac{1}{2}\, e^{u-1}\, n!\, (\log 2)^{-(n+1)}$; dividing by the case $u = 1$, whose coefficients are $F_n$, yields
\[
\frac{1}{F_n} \sum_{\alpha \in \UPF_n} u^{\fix(\alpha)} \longrightarrow e^{\,u-1} = \sum_{r \ge 0} \frac{e^{-1}}{r!}\, u^r,
\]
the probability generating function of a Poisson variable of mean $1$. Pointwise convergence of probability generating functions on $[0,1]$ implies convergence in distribution for nonnegative integer random variables~\cite[Ch.~IX]{flajolet2009}.
\end{proof}

\subsection{The fixed-block decomposition}

The Fubini and deranged Bell numbers differ only in whether the block permutation may have fixed points. Making this precise gives a convolution with the Bell numbers.

\begin{theorem}[Fixed-block decomposition]\label{thm:fubrel}
For all $n \ge 0$,
\begin{equation}\label{eq:fubrel}
F_n = \sum_{j=0}^{n} \binom{n}{j} B_j\,\tilde F_{n-j};
\end{equation}
equivalently, $F(x) = B(x)\,\tilde F(x)$.
\end{theorem}

At the level of generating functions the identity is one line: $B(x)\tilde F(x) = e^{e^x-1} \cdot e^{1-e^x}/(2-e^x) = F(x)$. The point of the proof below is the bijection, which splits every unit-interval parking function into its fixed part and its deranged part.

\begin{proof}
We give a bijection between $\OSP_n$ and triples $(S, P, Q)$, where $S \subseteq [n]$, $P$ is an (unordered) set partition of $S$, and $Q$ is a deranged ordered set partition of $[n] \setminus S$. Summing over $|S| = j$, the triples number $\sum_{j} \binom{n}{j} B_j \tilde F_{n-j}$, while $|\OSP_n| = F_n$; through $\phi$, the statement transfers to $\UPF_n$.

\emph{Forward map.} Given $(C_1, \dots, C_k) \in \OSP_n$ with min-ordered blocks $B_1, \dots, B_k$ and block permutation $\sigma$ (so $C_i = B_{\sigma(i)}$), let $S$ be the union of the fixed blocks, $P$ the partition of $S$ into those fixed blocks, and $Q$ the subsequence of non-fixed blocks. Since $\sigma$ fixes each fixed position and is a bijection, it permutes the non-fixed positions among themselves with no fixed point; the non-fixed blocks keep their relative min-element order inside $[n] \setminus S$, so $Q$ is a deranged ordered set partition of $[n] \setminus S$.

\emph{Inverse map.} Given $(S, P, Q)$, list all blocks of $P$ and $Q$ together by increasing minimum element as $D_1, \dots, D_k$; place each block of $P$ at its own canonical position, and fill the remaining positions, in increasing order, with the blocks of $Q$ in their deranged order. The $P$-blocks are then fixed. For the $Q$-blocks, let $a'_1 < a'_2 < \dots$ enumerate the canonical positions not used by $P$ (these coincide with the internal min-order of $Q$). If $Q$ places at its $t$-th position the block of $Q$-rank $\tau(t)$, with $\tau$ a derangement, then the full list places $D_{a'_{\tau(t)}}$ at position $a'_t$; since $t \mapsto a'_t$ is strictly increasing, $a'_t = a'_{\tau(t)}$ would force $t = \tau(t)$, which is impossible. So no $Q$-block is fixed, the fixed set of the result is exactly $S$, and the two maps are mutually inverse.
\end{proof}

\begin{corollary}\label{cor:invfub}
For all $n \ge 0$,
\[
\tilde F_n = \sum_{j=0}^{n} \binom{n}{j}\,\cBell_j\,F_{n-j},
\qquad
\cBell_j := \sum_{k=0}^{j} (-1)^k \Stir{j}{k},
\]
where $\cBell_j$ are the complementary Bell (Uppuluri--Carpenter) numbers, with generating function $\sum_{j\ge0}\cBell_j x^j/j! = e^{\,1-e^x}$.
\end{corollary}

\begin{proof}
By Theorems~\ref{thm:egf} and~\ref{thm:fubrel}, $\tilde F(x) = F(x)/B(x) = e^{1-e^x}\,F(x)$; equate coefficients. Equivalently, this is the binomial inversion of~\eqref{eq:fubrel}, since $\sum_{j=0}^{n}\binom{n}{j}\cBell_j B_{n-j} = [\,n=0\,]$.
\end{proof}

\begin{remark}
Read through $\phi$, Theorem~\ref{thm:fubrel} says: every $\alpha \in \UPF_n$ is obtained by choosing the set $S$ of cars lying in its fixed blocks, an arbitrary set partition of $S$ into those fixed blocks ($B_{|S|}$ choices), and a deranged unit-interval parking function on the remaining cars. Convolutions of $\tilde F_n$ with $F_n$ and with itself appear in~\cite[Cor.~4.3, Cor.~4.5]{belbachir2021}, and the complementary Bell numbers here are consistent with the fixed-block analysis of the partial deranged Bell numbers in~\cite{djemmada2025}; we do not claim that the identity~\eqref{eq:fubrel} itself is essentially new, only the bijection and its parking reading.
\end{remark}

\subsection{Marking block sizes and singleton blocks}

The composition in the proof of Theorem~\ref{thm:egf} accepts a full set of block-size markers.

\begin{theorem}[Marking block sizes]\label{thm:multi}
Let $y_1, y_2, \dots$ be commuting variables and give each $\alpha \in \DUPF_n$ the weight $y(\alpha) := \prod_{j=1}^{m(\alpha)} y_{\ell_j}$, where $\ell_1, \dots, \ell_m$ are its block lengths. Then, as formal power series,
\[
\sum_{n \ge 0} \Bigl( \sum_{\alpha \in \DUPF_n} y(\alpha) \Bigr) \frac{x^n}{n!}
= D\Bigl( \sum_{j \ge 1} y_j\, \frac{x^j}{j!} \Bigr),
\qquad D(z) = \frac{e^{-z}}{1-z}.
\]
\end{theorem}

\begin{proof}
Repeat the proof of Theorem~\ref{thm:egf}, marking each block of size $j$ with $y_j$: the generating function of a marked nonempty set is $\sum_{j\ge1} y_j x^j/j!$, and the derangement structure on the set of blocks is unchanged.
\end{proof}

\begin{corollary}[Singleton blocks]\label{cor:singletons}
For $s \ge 0$, let $\DUPF_n^{(s)}$ be the set of $\alpha \in \DUPF_n$ with exactly $s$ blocks of size $1$. Then
\[
\sum_{n \ge 0} |\DUPF_n^{(s)}|\, \frac{x^n}{n!}
= \frac{x^s}{s!} \sum_{m \ge s} d_m\, \frac{(e^x - 1 - x)^{\,m-s}}{(m-s)!},
\qquad
|\DUPF_n^{(s)}| = \sum_{m=s}^{n} d_m \binom{n}{s} \StirTwo{n-s}{m-s},
\]
where $\StirTwo{N}{K}$, the $2$-associated Stirling number of the second kind, counts the partitions of $[N]$ into $K$ blocks of size at least $2$, with $\sum_{N\ge0} \StirTwo{N}{K}\, x^N/N! = (e^x - 1 - x)^K/K!$.
\end{corollary}

\begin{proof}
Set $y_1 = u$ and $y_j = 1$ for $j \ge 2$ in Theorem~\ref{thm:multi}; then $\sum_j y_j x^j/j! = ux + (e^x - 1 - x)$ and
\[
\sum_{n, s \ge 0} |\DUPF_n^{(s)}|\, u^s\, \frac{x^n}{n!}
= D\bigl(ux + e^x - 1 - x\bigr)
= \sum_{m \ge 0} \frac{d_m}{m!}\, \bigl(ux + e^x - 1 - x\bigr)^m.
\]
Expanding by the binomial theorem and extracting the coefficient of $u^s$ gives the stated series; comparing coefficients of $x^n/n!$ turns the factor $x^s/s!$ into $\binom{n}{s}$ and gives the counting formula. Directly: choose the $s$ elements that form singleton blocks, partition the remaining $n - s$ elements into $m - s$ blocks of size at least $2$, and choose a deranged ordering of all $m$ blocks.
\end{proof}

\begin{remark}\label{rmk:s0}
Summing over $s \ge 0$ (that is, setting $u = 1$) recovers Proposition~\ref{prop:count}. The case $s = 0$ counts the deranged unit-interval parking functions in which \emph{no block has size $1$}. Note that this does not mean that every car is displaced: by Section~\ref{sec:lucky}, every block, whatever its size, contains exactly one lucky car, namely its leader. The correct reading of $s = 0$ is that every lucky car shares its block with at least one displaced car.
\end{remark}

\section{Extensions and equivalent models}\label{sec:extensions}

\subsection{An \texorpdfstring{$r$}{r}-start deranged extension}\label{sec:r}

Bradt et al.~\cite{bradt2024} study the unit-interval parking functions of length $n+r$ whose first $r$ entries are pairwise distinct, written $\UPF^r_{n+r}$, and show that $|\UPF^r_{n+r}| = \sum_{k\ge0}(k+r)!\,\rStir{n+r}{k+r}{r}$, the $r$-Fubini number. Here $\rStir{n+r}{k+r}{r}$ is an $r$-Stirling number of the second kind~\cite{broder1984}, counting the partitions of $[n+r]$ into $k+r$ blocks with $1, \dots, r$ in distinct blocks. The next lemma is the exact form of the restriction of $\phi$ that we need.

\begin{lemma}\label{lem:rstart}
Let $\alpha \in \UPF_{n+r}$. The entries $a_1, \dots, a_r$ are pairwise distinct if and only if the cars $1, \dots, r$ lie in pairwise distinct blocks of $\phi(\alpha)$.
\end{lemma}

\begin{proof}
Suppose $a_1, \dots, a_r$ are pairwise distinct. By induction on $i \le r$: when car $i$ arrives, the occupied spots are exactly $a_1, \dots, a_{i-1}$, and $a_i$ differs from all of them, so car $i$ parks in $a_i$. Hence cars $1, \dots, r$ are all lucky; each lucky car is the leader of its block (Section~\ref{sec:lucky}), and each block has only one leader, so cars $1, \dots, r$ lie in pairwise distinct blocks.

Conversely, suppose cars $1, \dots, r$ lie in pairwise distinct blocks, and fix $i \le r$. Every car of smaller index lies in $\{1, \dots, i-1\} \subseteq \{1, \dots, r\}$, hence in a different block; so car $i$ has the smallest index in its block, i.e.\ it is the leader, hence lucky, and $a_i$ is the start of its block (Lemma~\ref{lem:psi}). Distinct blocks have distinct starts, so $a_1, \dots, a_r$ are pairwise distinct.
\end{proof}

\begin{definition}\label{def:rdupf}
For $r \ge 1$ and $n \ge 0$, the set of \emph{$r$-start deranged unit-interval parking functions} is
\[
\DUPF^r_{n+r} := \{\, \alpha \in \UPF^r_{n+r} : \phi(\alpha) \in \DOSP_{n+r} \,\} = \UPF^r_{n+r} \cap \DUPF_{n+r}.
\]
\end{definition}

\begin{theorem}\label{thm:rcount}
For all $r \ge 1$ and $n \ge 0$,
\[
|\DUPF^r_{n+r}| = \sum_{k\ge0} d_{k+r}\,\rStir{n+r}{k+r}{r}.
\]
\end{theorem}

\begin{proof}
By Lemma~\ref{lem:rstart} and Theorem~\ref{thm:phi-bij}, $\phi$ restricts to a bijection from $\UPF^r_{n+r}$ onto the ordered set partitions of $[n+r]$ in which $1, \dots, r$ lie in distinct blocks, and hence from $\DUPF^r_{n+r}$ onto the deranged such partitions. To build one with $k+r$ blocks, choose the underlying partition in $\rStir{n+r}{k+r}{r}$ ways, then a deranged ordering of all $k+r$ blocks in $d_{k+r}$ ways, as in the proof of Theorem~\ref{thm:strat}. Summing over $k$ gives the formula. For $r = 1$ the condition on the first entries is empty, and the formula reduces to $\tilde F_{n+1}$, consistent with $\UPF^1_{n+1} = \UPF_{n+1}$.
\end{proof}

Using $\sum_{n\ge0}\rStir{n+r}{k+r}{r}\frac{x^n}{n!} = \frac{e^{rx}(e^x-1)^k}{k!}$~\cite{broder1984,bradt2024} and $D^{(r)}(z) = \sum_{k\ge0} d_{k+r}\frac{z^k}{k!}$ (differentiate $D(z) = \sum_m d_m z^m/m!$ term by term), Theorem~\ref{thm:rcount} packages into the exponential generating function
\begin{equation}\label{eq:regf}
G_r(x) := \sum_{n\ge0} |\DUPF^r_{n+r}|\,\frac{x^n}{n!}
= e^{rx}\,D^{(r)}\!\bigl(e^x - 1\bigr),
\qquad D(z) = \frac{e^{-z}}{1-z}.
\end{equation}
The first values are
\[
|\DUPF^2_{n+2}| = 1,\ 4,\ 23,\ 171,\ 1522, \ldots,
\qquad
|\DUPF^3_{n+3}| = 2,\ 15,\ 125,\ 1180,\ 12629, \ldots
\]
for $n = 0, 1, 2, \ldots$ (confirmed by exhaustive search through length $7$).

\begin{proposition}[Asymptotics]\label{prop:rasym}
Fix $r \ge 1$. As $n \to \infty$,
\[
|\DUPF^r_{n+r}| \sim \frac{(n+r)!}{2e\,(\log 2)^{\,n+r+1}} \sim \tilde F_{n+r}.
\]
In particular $|\DUPF^r_N| / \tilde F_N \to 1$: for large $N$, almost every deranged unit-interval parking function of length $N$ has pairwise distinct first $r$ entries.
\end{proposition}

\begin{proof}
From $e^{-z} = e^{-1} \sum_{k\ge0} (1-z)^k/k!$ we get $D(z) = e^{-1}(1-z)^{-1} + (\text{a function analytic at } z = 1)$, hence
\[
D^{(r)}(z) = e^{-1}\, r!\, (1-z)^{-(r+1)}\bigl(1 + O(1-z)\bigr) \qquad (z \to 1).
\]
The function $G_r$ of~\eqref{eq:regf} is analytic where $e^x \neq 2$; its dominant singularity is $x_0 = \log 2$ (all other solutions of $e^x = 2$ have larger modulus), a pole of order $r+1$. Near $x_0$ we have $e^{rx} \to 2^r$ and $1 - (e^x - 1) = 2 - e^x = 2(x_0 - x)\bigl(1 + O(x_0 - x)\bigr)$, so
\[
G_r(x) \sim 2^r \cdot e^{-1} r!\,\bigl(2(x_0 - x)\bigr)^{-(r+1)} = \frac{r!}{2e}\,(x_0 - x)^{-(r+1)}.
\]
Singularity analysis~\cite[Ch.~IV--VI]{flajolet2009} then gives
\[
\frac{|\DUPF^r_{n+r}|}{n!} = [x^n]\,G_r(x) \sim \frac{r!}{2e}\binom{n+r}{r}\frac{1}{(\log 2)^{\,n+r+1}},
\]
and $n!\,r!\binom{n+r}{r} = (n+r)!$ gives the first claim. The second follows from $\tilde F_N \sim N!/\bigl(2e(\log 2)^{N+1}\bigr)$ (Remark~\ref{rmk:asym}).
\end{proof}

The convergence is slow; for instance $|\DUPF^2_6| = 1522$ against the limiting approximation $\approx 1723$.

\begin{remark}\label{rmk:not-belbachir}
The numbers of Theorem~\ref{thm:rcount} should not be confused with the \emph{$r$-deranged Bell numbers} $\tilde F_{n,r}$ of Belbachir et al.~\cite[\S6]{belbachir2021}, which equal $\sum_{k} d_{k,r}\rStir{n+r}{k+r}{r}$ with $d_{k,r}$ an $r$-derangement number (fixed-point-free permutations of $[k+r]$ in which $1, \dots, r$ occupy distinct cycles). Our condition deranges the \emph{full} block permutation with no constraint on cycle type, so $d_{k+r}$ replaces $d_{k,r}$, and the two sequences differ already at small parameters: our $|\DUPF^2_{n+2}|$ begins $1, 4, 23, 171$, whereas $\tilde F_{n,2}$ begins $2, 30, 362, 4390$. Thus $\DUPF^r_{n+r}$ is the ``fully deranged'' partner of the $r$-Fubini family of~\cite{bradt2024}: a sibling of, but distinct from, the $r$-deranged Bell numbers.
\end{remark}

\subsection{Connection to deranged Cayley permutations}\label{sec:cayley}

The bijection $\phi$ realizes $\DUPF_n$ inside ordered set partitions. We transport it one step further, to the sequence model of Mor and Fraenkel~\cite{morfraenkel1984}, where the deranged condition reads off the order of \emph{first appearances}.

\begin{definition}\label{def:cayley}
A \emph{Cayley permutation} of length $n$ is a sequence $p = (p_1, \dots, p_n) \in \mathbb{Z}_{>0}^{\,n}$ whose set of values is an initial segment $\{1, 2, \dots, k\}$, each value occurring at least once. We write $\Cay_n$ for the set of Cayley permutations of length $n$.
\end{definition}

Mor and Fraenkel~\cite{morfraenkel1984} showed $|\Cay_n| = F_n$. The underlying bijection with ordered set partitions is standard: a Cayley permutation $p$ with values $\{1, \dots, k\}$ corresponds to the ordered set partition $\rho(p) := (C_1, \dots, C_k)$, ordered by value, where $C_v = \{\, i \in [n] : p_i = v \,\}$; conversely, $(C_1, \dots, C_k)$ gives the Cayley permutation with $p_i = v \iff i \in C_v$. Thus $\rho \colon \Cay_n \to \OSP_n$ is a bijection.

For $p \in \Cay_n$ with values $\{1, \dots, k\}$, let $m_v := \min\{\, i : p_i = v \,\}$ be the position of the \emph{first occurrence} of $v$. Reading $p$ from left to right, the values make their first appearances in some order, recorded by the \emph{first-appearance permutation} $\theta_p := \std(m_1, \dots, m_k) \in S_k$: here $\theta_p(v) = j$ means that $v$ is the $j$-th distinct value to appear.

\begin{definition}\label{def:dcp}
A \emph{deranged Cayley permutation} of length $n$ is a Cayley permutation $p$ whose first-appearance permutation $\theta_p$ is a derangement: no value $v$ is the $v$-th distinct value to make its first appearance. We write $\DCP_n$ for this set.
\end{definition}

\begin{remark}\label{rmk:naive-cayley}
One might instead impose the simpler condition that the first occurrence of each value $v$ avoid \emph{position} $v$, that is, $m_v \neq v$ for all $v$. This is strictly weaker and does not produce the deranged Bell numbers: it admits $31$ sequences of length $4$ against $\tilde F_4 = 28$ (both counts verified by exhaustive search). For example, $p = (2,1,1,3)$ has first occurrences $(m_1, m_2, m_3) = (2,1,4)$, all with $m_v \neq v$; yet value $3$ is the third distinct value to appear, so $\theta_p(3) = 3$ is a fixed point and $p \notin \DCP_4$. The correct condition compares the \emph{rank} of the first occurrence, not its position, with $v$---exactly as the deranged condition on ordered set partitions compares a block against the min-element order rather than against absolute position.
\end{remark}

\begin{proposition}[Deranged Cayley permutations]\label{prop:cayley}
Let $\Theta := \rho^{-1} \circ \phi \colon \UPF_n \to \Cay_n$. Then $\Theta$ restricts to a bijection $\DUPF_n \to \DCP_n$. In particular, $|\DCP_n| = \tilde F_n$.
\end{proposition}

\begin{proof}
Both $\phi$ and $\rho$ are bijections, so $\Theta$ is one. Take $\alpha \in \UPF_n$ with $\phi(\alpha) = (I_1, \dots, I_m)$ and set $p = \Theta(\alpha)$, so that $p_i = v$ if and only if $i \in I_v$. The first occurrence of $v$ in $p$ is
\[
m_v = \min\{\, i \in [n] : p_i = v \,\} = \min I_v = \mu_v,
\]
the leader of the $v$-th value-block. Hence $(m_1, \dots, m_m) = W(\alpha)$ and $\theta_p = \std(W(\alpha))$. By Proposition~\ref{prop:std}, $\alpha \in \DUPF_n$ if and only if $\std(W(\alpha))$ is a derangement, which by Definition~\ref{def:dcp} is exactly $p \in \DCP_n$. The count follows from Proposition~\ref{prop:count}.
\end{proof}

The following diagram summarizes the three models and their deranged versions; all maps are the bijections defined above, and each restricts to the deranged objects.
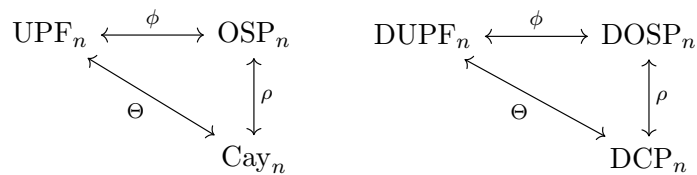
\[
\begin{tikzcd}[row sep=large, column sep=large]
\UPF_n \arrow[r, leftrightarrow, "\phi"] \arrow[dr, leftrightarrow, "\Theta"'] & \OSP_n \arrow[d, leftrightarrow, "\rho"] \\
& \Cay_n
\end{tikzcd}
\qquad
\begin{tikzcd}[row sep=large, column sep=large]
\DUPF_n \arrow[r, leftrightarrow, "\phi"] \arrow[dr, leftrightarrow, "\Theta"'] & \DOSP_n \arrow[d, leftrightarrow, "\rho"] \\
& \DCP_n
\end{tikzcd}
\]
The three models $\DUPF_n$, $\DOSP_n$, and $\DCP_n$ are the same object seen through cars, through blocks, and through first appearances, each counted by $\tilde F_n$.

\section{Concluding remarks}\label{sec:conc}

We close with one direction in which the transported structure might support results that are genuinely parking-theoretic, beyond the present note.

\begin{question}[$q$-analogues]\label{q:q}
Displacement is not a rich statistic here, since $D(\alpha) = n - m(\alpha)$ only sees the number of blocks. Natural finer statistics are the inversion number of the Cayley permutation $\Theta(\alpha)$, or the preference sum $a_1 + \dots + a_n$. Is there a closed form, say through $q$-Stirling numbers and a compatible $q$-analogue of the derangement numbers, for
\[
\sum_{\alpha \in \DUPF_n} q^{\operatorname{inv}(\Theta(\alpha))}
\qquad\text{or}\qquad
\sum_{\alpha \in \DUPF_n} q^{\,a_1 + \dots + a_n}\,?
\]
\end{question}

\end{document}